\documentclass[12pt]{amsart}
\usepackage{amssymb,amsmath,amstext,graphicx}

\theoremstyle{plain}

\newtheorem{theorem}{Theorem}[section]

\newtheorem{proposition}[theorem]{Proposition}
\newtheorem{lemma}[theorem]{Lemma}

\newtheorem{corollary}[theorem]{Corollary}

\theoremstyle{definition}

\theoremstyle{remark}

\newcommand{\I}{\mathcal{I}}

\newcommand{\T}{\mathcal{T}}
\newcommand{\R}{\mathbb{R}}

\newcommand{\supp}{\text{supp}}

\newcommand{\free}{\text{empty}}

\newcommand{\rank}{\text{rank}}

\begin{document}

\title[Fair division with multiple pieces]{Fair division with multiple pieces}

 \author{Kathryn Nyman}\thanks{}
\address{Kathryn Nyman: Willamette University} \email{knyman@willamette.edu}
\author{Francis Edward Su}\thanks{}
\address{Francis Su: Harvey Mudd College}\email{su@math.hmc.edu}
\author{Shira Zerbib}
\address{Shira Zerbib: University of Michigan and MSRI} \email{zerbib@umich.edu}


\begin{abstract}
Given a set of $p$ players we consider problems concerning envy-free allocation of collections of $k$ pieces from a given set of goods or chores. We show that if $p\le n$ and each player can choose $k$ pieces out of $n$ pieces of a cake, then there exist a division of the cake and an allocation of the pieces where at least $\frac{p}{2(k^2-k+1)}$ players get their desired $k$ pieces each. We further show that if $p\le k(n-1)+1$ and each player can choose $k$ pieces, one from each of $k$ cakes that are divided into $n$ pieces each, then there exist a division of the cakes and allocation of the pieces where at least $\frac{p}{2k(k-1)}$ players get their desired $k$ pieces. 
Finally we prove that if $p\ge k(n-1)+1$ and each player can choose one shift in each of $k$ days that are partitioned into $n$ shifts each, then, given that the salaries of the players are fixed, there exist $n(1+\ln k)$ players covering all the shifts, and moreover, if $k=2$ then $n$ players suffice. Our proofs combine topological methods and theorems of F\"uredi, Lov\'asz and Gallai from hypergraph theory.  
\end{abstract}

\maketitle

\section{Introduction}

Consider a group of people who are interested in sharing a house and office space. Is there a way to allocate rent on the rooms of the house and office building in such a way that all of the rooms will be rented, and each renter will get their first choice of room in each building?
And given an employer with a pool of potential workers, can they guarantee a small number of workers can be hired to cover all necessary shifts? 

In this paper, we examine several fair division questions pertaining to the allocation of multiple pieces of goods (or bads) to players. Our goal will be to divide multiple goods fairly among a group of people: we seek to divide multiple cakes such that players receive one piece from each cake, we wish to assign shifts on multiple days to workers so that all shift are covered, and we seek to assign rents to rooms in multiple buildings so that players prefer disjoint rooms in each building. We also consider the problem of assigning multiple pieces of a single cake to players. 

Since it is not always possible in these multi-piece allocation problems to make all players in an arbitrary group happy, our results assume a larger initial set of ``potential players'' and guarantee that some fraction of them can be satisfied with a distribution of pieces. In all of our problems, we seek to allocate pieces to players in such a way that they each receive their most preferred set of pieces in a given division, chosen over the set of all possible collections of pieces in that division. This makes the allocation
\emph{envy-free}, as the player would not wish to trade pieces with any other player. It also makes the allocation \emph{Pareto-optimal}, as no shuffling of the pieces among players would make any of the players happier (as they all have their top choice of pieces).

The problem of dividing a single cake among a group of players has been extensively studied (see, e.g. \cite{bramstaylor, robertsonwebb}). More recently, the division of multiple goods or bads has been examined. Cloutier, et.al.  \cite{multicakes} showed that when dividing two cakes into two pieces each, it is not always possible to satisfy the preferences of the two players if each of them chooses one piece of each cake.  
However, if the number of players is increased, or the number of pieces is increased, then it can be guaranteed that there are two players whose choices of pieces in each cake are disjoint. Hence, both these players can be satisfied simultaneously. Note that in this problem, if the number of players is increased,  
or the number of pieces is increased
(but both are not increased together), then there is either a player who does not receive any cake, or a piece of cake that has not been distributed to any player. We can consider these ``left out'' players or pieces as being the price of a disjoint, envy-free distribution of cake, as it may be impossible to satisfy two players otherwise. 

Lebert~et.~al.~\cite{Meunier} extended the work in \cite{multicakes} to show an upper bound on the number of pieces necessary to guarantee that $m$ cakes can be divided in an envy-free manner among two players. Here again, each of the two players receives a disjoint, envy-free set of pieces of the cakes, but there are extra pieces not allocated to any player.

Dual to the notion of dividing goods, one can ask about dividing ``bads'', or chores, which leads to the problem of rental division. In \cite{Rental}, Su showed that a division of rent among rooms in a house can be achieved so that each roommate prefers a different room. In a surprising result, Frick~et.~al.~\cite{SecretHousemate} show that an envy-free rent division can be achieved among $n$ people, even if the preferences of only $n-1$ housemates are known. 

However, when we extend the question to division of rent in multiple houses (say a bedroom house and an office building that are being rented to a set of roommates together), then, as in the case of dividing multiple cakes, it may be impossible to make every player happy simultaneously. Already in the case of two players and two houses with two rooms each, an example where there is no envy-free division of the rent is known  
(see Theorem 3 of \cite{multicakes}). 


We now summarize our results. Our proofs are topological, and combine extensions of KKM-type theorems with tools from hypergraph theory. 

Our first theorem gives a lower bound on the number of players who prefer mutually disjoint pieces in the case where one cake is divided into $n$ pieces and players chooses $k$ pieces each.  We imagine a set of potential players who, when presented with some division of the cake, have a preferred set of $k$ pieces. We strive for divisions in which we can maximize the number of players whose preferred $k$ pieces are disjoint, since those players can be satisfied simultaneously.

\begin{theorem}\label{main2}
Suppose a cake is to be divided into $n$ pieces, 
and there are $p\le n$ hungry players who satisfy the following conditions
\begin{enumerate}
\item in any division of cake into $n$ pieces, each player finds some subset of $k$ pieces acceptable, and
\item player preference sets are closed: a piece that is acceptable for a convergent sequence of divisions will also be acceptable in the limiting division.
\end{enumerate}
Then there exists a division of the cake into $n$ pieces where at least $\lceil\frac{p}{2(k^2-k+1)}\rceil$ players prefer mutually disjoint sets of $k$ pieces.
Moreover, if $p$ divides $n$ then there exists a division of the cake into $n$ pieces where at least $\lceil\frac{p}{k^2-k+1}\rceil$ players prefer mutually disjoint sets of $k$ pieces.
\end{theorem}

This theorem is proved in Section \ref{sec:1cake}.

Our second theorem involves dividing $k$ cakes into $n$ pieces each. We call a choice of one piece in each cake a \emph{$k$-piece selection}.  

\begin{theorem}\label{main1}
Suppose that there are $k$ cakes and
$p\le k(n-1)+1$ hungry players, each of whom finds at least one $k$-piece selection acceptable in any division of the cakes into $n$ pieces each. If all player preference sets are closed then there exists a division of the $k$ cakes into $n$ pieces each such that at least $\lceil\frac{p}{2k(k-1)}\rceil$ players prefer mutually disjoint $k$-piece selections.  Moreover, if $p$ divides $k(n-1)+1$ then there exists a division where at least $\lceil\frac{p}{k(k-1)}\rceil$ players prefer mutually disjoint $k$-piece selections. 
\end{theorem}

We prove this Theorem in Section \ref{sec:kcakes}.

Our final fair division result gives conditions under which an employer can cover all necessary work shifts on $k$ days with a small number of employees. In this scenario, an employer seeks to divide $k$ work days into $n$ shifts each in such a way that s/he is able to cover all of the shifts with a limited number of employees. We prove the following.

\begin{theorem}\label{main3}
Suppose that a set of $p\ge k(n-1)+1$ employees satisfies the following conditions: 
\begin{enumerate}
\item For any partition of $k$ days into $n$ shifts each, every employee finds at least one selection of $k$ shifts, one in each day, acceptable.
\item The employees prefer empty shifts, if available. 
\item The employees' preference sets are closed.   
\end{enumerate}
Then there exists a partition of $k$ days into $n$ shifts each for which a subset of at most $n(1+\ln k)$ employees cover all the $kn$ shifts. Moreover, if $k=2$ there exists a subset of $n$ employees that cover all the $2n$ shifts. 
\end{theorem}

This result is an extension of Theorem 1.5 of \cite{AHZ} which can be interpreted to define the preferences of potential employees more narrowly. There one imagines a scenario in which potential employees submit in advance a personal time schedule, consisting of one time interval per day, during which s/he is willing work. Our Theorem \ref{main3} generalizes this result to every model of employee shift preferences, where  employee daily schedules are not necessarily single intervals, nor predetermined.

For the proof of Theorem \ref{main3} we prove a combinatorial dual extension of the topological KKMS theorem of Shapley \cite{shapley} to products of simplices, where dualization happens in each factor. Let $\Delta_{n-1}$ be the $(n-1)$-dimensional simplex and let $P(n,k)=(\Delta_{n-1})^k$ be the $k(n-1)$-dimensional polytope obtained by taking the Cartesian product of $k$ copies of $\Delta_{n-1}$. 
Let $T$ be a triangulation of $P(n,k)$. We say that a function  
$\ell: V(T) \rightarrow [n]^k$ 
is a \emph{factorwise-dual-Sperner labeling} of $T$ if for every $v=(v_1,\dots,v_k)\in V(T)$ and for every $1\le i \le k$,   
if $[n]\setminus \supp(v_i) \neq \emptyset$, 
we have $\ell(v)_i \in [n]\setminus \supp(v_i)$, where $\supp(v_i)$ is the minimal face of $\Delta_{n-1}$ containing $v_i$. 
We further say that two factorwise-dual-Sperner labelings 
$\ell, \ell': V(T) \rightarrow [n]^k$ 
are \emph{equivalent} if $\ell(v)_i=\ell'(v)_i$
whenever $[n]\setminus \supp(v_i) = \emptyset$. 
We prove: 

\begin{theorem}\label{dualPsperner}
If $T$ is a triangulation of $P(n,k)$ with a factorwise-dual-Sperner labeling $\ell$, then it has an equivalent labeling $\ell'$ with an elementary simplex $Q$ in $T$, such that the set of labels 
$\Lambda(Q)=\{\ell'(v) \mid v\in V(Q)\}$ 
is balanced with respect to $V=V_1\sqcup \dots \sqcup V_k$, where 
$V_i=[n]$ for all $1\le i\le k$. 
\end{theorem} 
Here ``balanced" means that one can find non-negative weights on $\Lambda(Q)$ such that the sum of weights in each component is 1. 
Theorems \ref{main3} and \ref{dualPsperner} are proved in Section 6. 

A Corollary of Theorem \ref{main3} guarantees an envy-free way to rent out all of the rooms in two $n$-room buildings, given an adequate pool of potential renters. Consider a set of colleagues, each of whom prefers free 
rooms if available, has closed preference sets, and in any division of the rents finds a collection of two rooms, one in each building, acceptable. 
 \begin{corollary}
If $2n-1$ such players seek to rent two rooms, one in each of two buildings containing $n$ rooms each, then there exists a division 
of rents in which a subset of $n$ players each get their preferred two rooms. 
\end{corollary}

\section{Preliminaries in hypergraph theory}
A {\em hypergraph} is a pair $H=(V,E)$ where $V=V(H)$ is a {\em vertex set}, and $E=E(H)$ is a finite collection of subsets of $V$ called {\em edges}.  
A hypergraph $H=(V,E)$ is {\em $k$-partite} if there exists a partition $V=V_1\cup\dots\cup V_k$ where every edge $e\in E$ has $|e\cap V_i|=1$ for all $1\le i\le k$. 
A $2$-partite hypergraph is a {\em bipartite graph}. 
The {\em rank} of a hypergraph $H=(V,E)$, denoted $\rank(H)$, is the maximum size of an edge in $E$, and the {\em degree} of a vertex $v\in V$, denoted $\deg(v)$, is the number of edges in $E$ containing $v$. 
A hypergraph $H=(V,E)$ is $k$-uniform if $|e|=k$ for all $e\in E$.    
 
A {\em matching} in a hypergraph $H=(V,E)$ is a set of disjoint edges. The \emph{matching number} $\nu(H)$ is the maximal size of a matching in $H$. A {\em fractional matching} of $H$ is a function $f:E\to [0,1] $, 
such that for every $v\in V$ we have $\sum_{e: v\in e} f(e) \le 1$. 
We can think of $f(e)$ as the \emph{weight} of edge $e$ in the formal sum 
$\sum_{e \in E} f(e) \cdot e$, where at each vertex $v$, the sum of coefficient of edges containing $v$ is at most $1$. 
The {\em fractional matching number} of $H$  is 
denoted by $\nu^*(H)$ and is defined as 
$$\nu^*(H)= \max\big\{\sum_{e\in E}f(e) \mid 
f\text{ is a fractional matching of } H\big\}.$$ 

A fractional matching $f:E(H)\to [0,1]$ is called {\em perfect} if for every $v\in V$ we have  
$\sum_{e: v\in e} f(e) = 1$. 
A collection of sets $E\subset 2^V$ is {\em balanced with respect to a set $V$} if the hypergraph $H=(V,E)$ has a perfect fractional matching. 

For example, consider the vertex set $V=\{1,2,3,\bar{1},\bar{2},\bar{3}\}$, with edge set comprised of all pairs of vertices with one barred and one un-barred vertex. Then the edge set $\{e_1=(1 \bar{1}),e_2=(2 \bar{1}),e_3=(1 \bar{2}),e_4=(2 \bar{2}),e_5=(3\bar{3})\}$ is balanced with respect to $V$ since there are edge weights that sum to $1$ at each vertex. In particular, 
 $\frac{1}{2}e_1 + \frac{1}{2}e_2+ \frac{1}{2}e_3 + \frac{1}{2}e_4+1e_5$ is a perfect fractional matching.

A {\em cover} of a hypergraph $H=(V,E)$ is a subset of vertices that intersect every edge. The \emph{covering number} $\tau(H)$ is the minimal size of a cover in $H$. A {\em fractional cover} of $H$ is a function $g:V\to [0,1] $, such that for every $e\in E$ we have $\sum_{v: v\in e} g(v) \ge 1$. The {\em fractional covering number} of $H$ is 
$$\tau^*(H)= \min\big\{\sum_{v\in V}g(v) \mid g  \text{ is a fractional covering of } H\big\}.$$ A {\em perfect fractional cover} of $H$ is a fractional cover with $\sum_{v: v\in e} g(v) = 1$ for every $e\in E$.
By linear programming duality, we have $\nu(H)\le \nu^*(H)=\tau^*(H) \le \tau(H)$ for every hypergraph $H$.

\begin{lemma}\label{rank}
If a hypergraph $H=(V,E)$ 
of rank $n$  has a perfect fractional matching, then $\nu^*(H)\ge \frac{|V|}{n}$. If, in addition, $H$ is $n$-uniform, then $\nu^*(H)= \frac{|V|}{n}$.  
\end{lemma}
\begin{proof}
Let $f:E\to [0,1]$ be a perfect fractional matching of $H$. Then 
\begin{equation}\label{eq1}
\sum_{v\in V}\sum_{e: v\in e} f(e) = \sum_{v\in V} 1 = |V|.
\end{equation} 
Since $f(e)$ was counted $|e|\le n$ times in (\ref{eq1}) for every edge $e\in E(H)$, we have that 
$$\nu^*(H)\ge \sum_{e\in E} f(e) \ge \frac{|V|}{n}.$$ 
If $H$ is $n$-uniform, then the constant function $g:V\to \{\frac{1}{n}\}$ is a fractional cover and therefore $\nu^*(H)=\tau^*(H)\le \frac{|V|}{n}$.  Combining with the inequality above, we have $\nu^*(H)=\frac{|V|}{n}$.
\end{proof}

We will use the following bounds on the ratio $\nu^*/\nu$ and $\tau/\tau^*$:
\begin{theorem}[F\"uredi \cite{furedi}]\label{furedi}
If $H$ is a hypergraph of rank $n\ge 2$, then
$\nu(H) \ge \frac{\nu^*(H)}{n-1+\frac{1}{n}}$. If $H$
is $n$-partite, then
$\nu(H) \ge \frac{\nu^*(H)}{n-1}$.
\end{theorem}

\begin{theorem}[Lov\'asz \cite{lovasz}]\label{lovasz}
If $H$ is a hypergraph with maximal degree $d$, then $\tau(H) \le (1+\ln
d)\tau^*(H)$.
\end{theorem}

The dual hypergraph $H^{D}=(U,F)$ of a hypergraph $H=(V,E)$ is obtained by reversing the roles of vertices and edges, namely, $U=E$, $F=V$, and an edge $v\in F$ consists of all the vertices $e \in U$ for which $v\in e$ in $H$. If $H$ is $d$-uniform then $\deg(e) = d$ for all $e \in U$ and we have the following corollary of Theorem \ref{lovasz}:  
\begin{corollary}\label{lovaszcor}
Let $H$ be a $d$-uniform hypergraph then $\tau(H^D) \le (1+\ln d)\tau^*(H^D)$.
\end{corollary}

If $H$ is a bipartite graph then the situation is even better:
\begin{theorem}[Gallai \cite{gal59}]\label{bipgallai}
If $H$ is a bipartite graph then $\tau(H^D) =\tau^*(H^D)$.
\end{theorem}

\section{Complete triangulations of polytopes}

Let $T$ be a  triangulation of a $(d-1)$-dimensional polytope $P$ and let $V(T)$ be the vertices ($0$-dimensional faces) of $T$. The $(d-1)$-dimensional faces of $T$ are called the {\em elementary simplices} of $T$. 
For each vertex $v\in V(T)$ we assign an owner $o_v\in [d]$ (where $[d]$ is thought of as a set of players). Such an assignment will be said to be {\em complete} if every elementary simplex $Q$ in $T$ has $\{o_v \mid v\in V(Q)\} = [d]$ and a triangulation will be called {\em complete} if it has a complete ownership assignment. 

\begin{lemma}\label{complete}
Every $d$-dimensional polytope $P$ has a complete triangulation. 
\end{lemma}
\begin{proof}
Let $T$ be any triangulation of $P$. Then the first subdivision $T'$ of $T$ is a complete triangulation of $P$, with the assignment $o_v = i$ if $v\in V(T')$ is the barycenter of an $(i-1)$-dimensional face of $T$.  \end{proof}

\section{Allocating multiple pieces of one cake to players}\label{sec:1cake}
In this section, we consider dividing one cake into $n$ pieces and assigning $k$ pieces to each player. We imagine that our cake is rectangular, and is divided by parallel vertical cuts. We represent the division of the cake into $n$ pieces by the $n$-tuple $\vec{x}=(x_1,x_2,
\ldots , x_{n})\in \mathbb{R}^{n}_{\geq 0}$ satisfying $\sum_{i=1}^{n}x_i=1$, where $x_i$ denotes the width of piece $i$. The space of all possible divisions of one cake into $n$ pieces can therefore be realized by the $(n-1)$-dimensional simplex $\Delta_{n-1}$.

The preferences of an individual player can be visualized as a covering of the space of divisions of the cake by sets corresponding to all possible collections of $k$ pieces. Each point in the space is assigned to the set corresponding to the $k$ pieces that the player would prefer if the cake were cut according to that division. We note that it is possible for a point to belong to multiple sets if the player is indifferent to more than one collection of $k$ pieces in the given division. 

We make the following natural assumptions on player preferences.
\begin{enumerate}
    \item The players are hungry: no player prefers an empty piece of cake to a non-empty piece.
    \item The players have closed preference sets: if a player prefers a given set of $k$ pieces in a sequence of divisions that approach a limit, then the player prefers that same set of $k$ pieces in the cut corresponding to the limit of the sequence of divisions.
    \end{enumerate}

For $t=(t_1,\dots,t_n)\in \R^n$ define the support of $t$ to be the set $\text{supp}(t)=\{i \mid t_i\neq 0\}.$
Our main tool in the proof of Theorem \ref{main2} is a generalization of Sperner's lemma, due to Shapley \cite{shapley}, in which vertices have subsets of labels (rather than single labels):

\begin{theorem}[Shapley \cite{shapley}] \label{discretekkms}
If $T$ is a triangulation of $\Delta_{n-1}$ with a labeling $L:V(T) \rightarrow 2^{[n]}$ such that $L(v) \subseteq \text{supp}(v)$,
then there exists an elementary simplex $Q$ in $T$ where the set 
$\{L(v) \mid v\in V(Q)\}$ is a balanced set with respect to $[n]$. 
\end{theorem}

Combining Theorems \ref{discretekkms} and \ref{furedi} we obtain the following.
\begin{proposition}\label{kkmsmatchings} 
Suppose that the conditions of Theorem \ref{discretekkms} hold, and in addition $|L(v)| \le k$ for every $v\in V(T)$. Then there exists an elementary simplex $Q$ in $T$, and 
a collection $M$ of $\frac{n}{k^2-k+1}$ vertices of $Q$, such that the sets  $\{L(v) \mid v\in M\}$ are pairwise disjoint.  
\end{proposition}
\begin{proof}
By Theorem \ref{discretekkms}, there exists an elementary simplex $Q$ in $T$ where the set 
$\T=\{L(v) \mid v\in V(Q)\}$ 
is a balanced set with respect to $[n]$. 
Therefore, the hypergraph $H=([n],\T)$ 
has a perfect fractional matching. Since every edge in $\T$ is of size at most $k$, by Lemma \ref{rank} we have that $\nu^*(H)\ge \frac{|V|}{k} = \frac{n}{k}.$ Therefore, by Theorem \ref{furedi}, $$\nu(H)\ge \frac{\nu^*(H)}{k-1+\frac{1}{k}} \ge \frac{n}{k^2-k+1}.$$
We conclude that $\T$ contains a matching $M'$ of size at least $\frac{n}{k^2-k+1}$, which corresponds to a subset $M$ of vertices of $Q$. 
\end{proof}

We are now ready to prove our first main result which we restate.
\medskip

\noindent {\bf Theorem \ref{main2}.} \emph{Suppose a cake is to be divided into $n$ pieces, 
and there are $p\le n$ hungry players who satisfy the following conditions
\begin{enumerate}
\item in any division of cake into $n$ pieces, each player finds some subset of $k$ pieces acceptable, and
\item player preference sets are closed: a piece that is acceptable for a convergent sequence of divisions will also be acceptable in the limiting division.
\end{enumerate}
Then there exists a division of the cake into $n$ pieces where at least $\lceil\frac{p}{2(k^2-k+1)}\rceil$ players prefer mutually disjoint sets of $k$ pieces.
Moreover, if $p$ divides $n$ then there exists a division of the cake into $n$ pieces where at least $\lceil\frac{p}{k^2-k+1}\rceil$ players prefer mutually disjoint sets of $k$ pieces.}

\begin{proof}
Let $S$ be the set of players with $|S|=p$. We duplicate each player $\lceil\frac{n}{p}\rceil$ times and choose a set $S'$ of size $n$ from the copies of players.
By Lemma \ref{complete} there exists a complete triangulation $T$ of $\Delta_{n-1}$. 
Let $o_v\in S'$ be the owner of $v\in V(T)$ in a complete assignment.

Every vertex $v\in V(T)$  corresponds to a division of the cake into $n$ pieces. 
Let the labeling $L(v)$ of $v$ be the preferred $k$-piece selection of $o_v$. 
Since $o_v$ is hungry, we have 
$L(v) \subseteq \text{supp}(v)$. 
Thus, by Theorem \ref{kkmsmatchings} there exists an elementary simplex $Q$ in $T$, with the following property: there exists a subset $M'$ of the vertices of $Q$ such that $|M'|\ge \lceil \frac{n}{k^2-k+1}\rceil$ and the set 
$\{L(v) \mid v\in M'\}$ 
consists of pairwise disjoint subsets of $[n]$ of size $k$ each, that represent pairwise disjoint $k$-piece selections  by $\lceil \frac{n}{k^2-k+1}\rceil$ players in $S'$. But since every player in $S$ is represented in $S'$ by at most $\lceil\frac{n}{p}\rceil$ copies and $p\le n$, there must be a set $M\subset M'$ representing pairwise disjoint $k$-piece selections by different players in $S$, with $$|M| \ge \frac{\lceil \frac{n}{k^2-k+1}\rceil}{\lceil\frac{n}{p}\rceil} \ge \frac{p}{2(k^2-k+1)},$$ and if $p$ divides $n$ then this improves to $|M| \ge \frac{p}{k^2-k+1}.$ 

Write $m = \lceil \frac{p}{k^2-k+1} \rceil$ if $p$ divides $n$, and $m= \lceil\frac{p}{2(k^2-k+1)} \rceil$ otherwise.
To show the existence of a single 
division of the cake that would satisfy at least $m$ players, carry out the procedure above for a sequence of finer and finer complete triangulations $T$. By compactness of $\Delta_{n-1}$ and decreasing size of the elementary simplices, there must exist a subsequence $\mathcal{Q}=Q_1,Q_2,\dots$ of elementary simplices converging to a single point, such that in each $Q_i$ there are at least $m$ players of $S$ that have pairwise disjoint $k$-piece selections. Since there are finitely many subsets of the player set $S$, and each player has a closed preference set, there must exists a subset $N\subset S$ of players of size $|N|=m$ for which the limit point of $\mathcal{Q}$ corresponds to a division of the cake 
in which the players in $N$ are satisfied with different sets of $k$ pieces.
\end{proof}

\section{Dividing multiple cakes}\label{sec:kcakes}
We now turn to the problem of dividing multiple cakes among a set of potential players.
 We will be dividing $k$ cakes into $n$ pieces each, and hence our space of all possible divisions will be realized by the $k(n-1)$-dimensional polytope $P(n,k)$ which is defined as the product of $k$ simplices of dimension $n-1$ each, that is, $$P(n,k) =  (\Delta_{n-1})^k = \Delta_{n-1} \times \dots \times \Delta_{n-1}.$$ Every face of $P(n,k)$ is given by $F(J_1,\dots,J_k)=\Delta^{J_1} \times \dots \times \Delta^{J_k}$ for some choice of subsets $J_1,\dots,J_k$ of $[n]$, where $\Delta^{S}$ is the face of $\Delta_{n-1}$ spanned by the vertices in $S\subset [n]$.

The KKMS theorem is a continuous version of Theorem \ref{discretekkms}. 
A theorem of Komiya \cite{komiya} implies the following generalization of the KKMS theorem to the polytope $P(n,k)$:
\begin{theorem}[Komiya \cite{komiya}]\label{komiyalite}  
Let $B_{(i_1,\dots,i_k)}$ be a closed subset of $P=P(n,k)$ for every $k$-tuple $(i_1,\dots,i_k) \in [n]^k$, such that for every $J_1,\dots, J_k \subseteq [n]$ we have $$F(J_1,\dots, J_k) \subseteq \bigcup \big\{B_{(i_1,\dots,i_k)}\mid i_j \in J_j \text{ for all } 1\le j\le k\big\}.$$
Then there exists a balanced collection of tuples  $\I=\big\{I^t= (i^t_1,\dots,i^t_k)
\in [n]^k \mid 1\le t\le m\big\}$ with respect to the  vertex set $V=V_1\sqcup\dots \sqcup V_k$, the disjoint union where 
$V_i=[n]$ for all $1\le i\le k$, such that $\bigcap_{t=1}^m B_{(i^t_1,\dots,i^t_n)} \neq \emptyset.$    
\end{theorem}

Suppose that the conditions of Theorem \ref{komiyalite} hold.
Then by the theorem, we obtain a $k$-partite hypergraph $H$ on vertex set $V=V_1\sqcup\dots \sqcup V_k$ with an edge set $\I=\big\{I^t= (i^t_1,\dots,i^t_k)
\in [n]^k \mid 1\le t\le m\big\}$ that 
has a perfect fractional matching. Since every edge in $\I$ is of size $k$, by Lemma \ref{rank} we have that $\nu^*(H)\ge \frac{|V|}{k} = \frac{kn}{k}=n.$ Therefore, by Theorem \ref{furedi}  we have $\nu(H)\ge \frac{\nu^*(H)}{k-1} \ge \frac{n}{k-1}.$
We conclude that $\I$ contains a matching $M$ of size at least $\frac{n}{k-1}$ and $\bigcap_{I^t\in M} B_{I^t} \neq \emptyset.$

Thus we proved:
\begin{proposition}\label{komiyamatchings}
Suppose that the conditions of Theorem \ref{komiyalite} hold. Then there exists a pairwise disjoint collection of $k$-tuples $$\big\{I^t = (i^t_1,\dots,i^t_k) \in [n]^k\mid 1\le t\le m\big\}$$ such that $m\ge \frac{n}{k-1}$ and $\bigcap_{t=1}^m B_{I^t} \neq \emptyset.$    
\end{proposition}

Here is a discrete version of Proposition \ref{komiyamatchings}:

\begin{proposition}\label{komiyadiscrete} 
Let $T$ be a triangulation of $P(n,k)$ with labeling $\ell:V(T) \rightarrow [n]^k$ such that if $v=(v_1,\dots,v_{k})$, then 
$\ell(v) = (\ell_1(v),\dots,\ell_k(v))$ where $\ell_i(v) \in \supp(v_i)$.
Then there exists an elementary simplex $Q$ in $T$, with the following property: there exists a subset $M$ of $V(Q)$ such that $|M|\ge \frac{n}{k-1}$ and the set $\{\ell(v) \mid v\in M\}$ consists of pairwise disjoint $k$-tuples.  
\end{proposition}
\begin{proof}
Let $T'$ be the first barycentric subdivision of $T$. For every $k$-tuple $I=(i_1,\dots,i_k) \in [n]^k$, we set 
$V_I = \{v\in T \mid L(v) = I\}$ 
and let 
$B_I$ be the union of all elementary simplices $S$ of $T'$ with $V(S)\cap V_I \neq \emptyset$. Then the sets $B_I$ satisfy the conditions of Proposition \ref{komiyamatchings}, so considering the conclusion, every point $x \in \bigcap_{t=1}^m B_{I^t}$ lies in simplex of $T$ that contains $M$ in its vertex set.  
\end{proof}

We are now ready to prove our second main theorem:
\medskip

\noindent {\bf Theorem \ref{main1}} \emph{
Suppose that there are $k$ cakes and
$p\le k(n-1)+1$ hungry players, each of whom finds at least one $k$-piece selection acceptable in any division of the cakes into $n$ pieces each. If all player preference sets are closed then there exists a division of the $k$ cakes into $n$ pieces each such that at least $\lceil\frac{p}{2k(k-1)}\rceil$ players prefer mutually disjoint $k$-piece selections.  Moreover, if $p$ divides $k(n-1)+1$ then there exists a division where at least $\lceil\frac{p}{k(k-1)}\rceil$ players prefer mutually disjoint $k$-piece selections. }

\begin{proof} 
Let $S$ be the set of players with $|S|=p$. As in the proof of Theorem \ref{main2}, we duplicate each player $\lceil\frac{k(n-1)+1}{p}\rceil$ times and choose a set $S'$ of size $k(n-1)+1$ from the copies of the players.
By Lemma \ref{complete}, there exists a complete triangulation $T$ of the  polytope $P=P(n,k)$. Let $o_v\in S'$ be the owner of $v\in V(T)$ in a complete assignment. 

Every vertex $v=(v_1,\dots,v_k) \in V(T)$ ($v_i \in \Delta_{n-1}$) corresponds to a division of the $k$ cakes into $n$ pieces each.
Define a labelling $\ell:V(T) \rightarrow [n]^k$ by setting $\ell(v)$ to be the $k$-tuple representing the $k$-piece selection, one from each cake, that the player $o_v$ prefers. 
Since $o_v$ is hungry, we have 
$\ell_i(v)\in \supp(v_i)$ for each $i \in [k]$.
Thus, by Proposition \ref{komiyadiscrete}, there exists an elementary simplex $Q$ in $T$, with the following property: there exists a subset $M'$ of the vertices of $Q$ such that $|M'|\ge \lceil \frac{n}{k-1}\rceil$ and the set 
$\{\ell(v) \mid v\in M'\}$ 
consists of pairwise disjoint $k$-tuples, that represent pairwise disjoint $k$-piece selections  
of $\lceil \frac{n}{k-1}\rceil$ players in $S'$. 

Now, since every player in $S$ is represented in $S'$ by at most $\lceil\frac{k(n-1)+1}{p}\rceil$ copies, there must be a set $M\subset M'$ of vertices in $Q$ owned by disjoint players in $S$, with
$$|M| \ge \frac{\lceil \frac{n}{k-1}\rceil}{\lceil\frac{k(n-1)+1}{p}\rceil} \ge \frac{p}{2k(k-1)},$$
and if $p$ divides $k(n-1)+1$ then $|M| \ge  \frac{p}{k(k-1)}.$ 

A similar argument to the one made in the proof of Theorem \ref{main2} now shows the existence of a single cut of the $k$ cakes into $n$ pieces each that would satisfy at least $\frac{p}{2(k^2-k)}$ players, or $\frac{p}{k(k-1)}$ players in the case $p$ divides $k(n-1)+1$. 
\end{proof}

We note that the techniques in this proof are similar to those used by Lebert, et. al. \cite{Meunier} to prove that if there are only two players then, for large enough $n$, there exists a division of the $k$ cakes into $n$ pieces each in which both players prefer pairwise disjoint $k$ pieces, one from each cake.


\section{Assigning Shifts to Players}\label{sec:shifts}

Our last allocation problem involves an employer who wishes to assign all shifts on a set of days to a collection of employees. We look for a small set of employees who can cover all of the shifts. This would be a useful consideration if, for example, the employees were all receiving fixed salaries, and so the employer wishes to spend the least amount of money to cover all of the shifts. Moreover, it is natural to assume (and we will make this assumption) that 
if the salaries are fixed 
then every employee prefers an empty shift (a shift that requires no time at work), if one is available. 

We also assume that if there are multiple empty shifts, then players are indifferent between them.  This is actually a consequence of having closed preference sets, since the division with multiple empty shifts is the limit of divisions in which a single fixed room is free.

We call a choice of one shift on each of $k$ days a \emph{$k$-shift selection}. 

We note that in a typical solution to our problem, more employees are needed than there are shifts each day. This indicates that some shifts will be covered by more than one employee, or some employees can be sent home early.

Consider a polytope $P$ with a triangulation $T$. A labeling of $T$ is said to be {\em Sperner} if the vertices of $P$ have distinct labels, and vertices of $T$ that lie on a minimal face $F$ of $P$ have labels chosen from the labels of the vertices of $P$ spanning the face $F$.
A simplex of $T$ is said to be a {\em full cell} if its vertices have distinct labels.
We will require the following result from \cite{Polytopal} which comes from a generalization of Sperner's Lemma to polytopes. 
\begin{theorem}[DeLoera-Peterson-Su \cite{Polytopal}] 
\label{polytopal}
Let $P$ be an $d$-dimensional polytope with $n$ vertices together with a Sperner-labeled triangulation $T$. Let $f: P \rightarrow P$ be the piecewise-linear map that takes each vertex of $T$ to the vertex of $P$ that shares the same label, and is linear on elementary simplex of $T$. Then the map $f$ is surjective, and thus the collection of full cells in $T$ forms a cover of $P$ under $f$.
\end{theorem}

Each point in the product of simplices $P(n,k)$ represents a division of $k$ days into $n$ shifts each. Let $T$ be a triangulation of $P(n,k)$.  We shall label every vertex $v =(v_1,\dots,v_k) \in V(T)$ (where $v_i \in \Delta_{n-1}$) with a $k$-tuple 
$\ell(v) = (\ell_1(v),\dots,\ell_k(v))$ in $[n]^k$
that will represent a desirable shift selection for some player.

We now define a labeling condition to describe the fact that players prefer empty shifts when available. For any 
$x \in \Delta_{n-1}$, let 
$\free(x) = [n] \setminus \supp(x)$.
We say a labeling
$\ell: V(T) \rightarrow [n]^k$ 
is \emph{factorwise-dual-Sperner} if 
whenever $\free(v_i) \neq \emptyset$, 
we have $\ell_i(v) \in \free(v_i)$.
Since players are indifferent between empty shifts when there are more than one, we define also a notion of equivalent labelings that allow for switching preferences between empty shifts.
We say two factorwise-dual-Sperner labelings 
$\ell, \ell': V(T) \rightarrow [n]^k$ 
are \emph{equivalent} if 
whenever $\free(v_i)= \emptyset$, we have $\ell_i(v)=\ell'_i(v)$. 

For the proof of Theorem \ref{main3} we need a dual-Sperner-type result for $P(n,k)$. This is Theorem \ref{dualPsperner} that we restate and prove here.
For its proof we employ a technique used by Frick, et. al. \cite{SecretHousemate} to give a Sperner labeling in their rental division problem with a secretive housemate.
\medskip

\noindent {\bf Theorem \ref{dualPsperner}} \emph{
If $T$ is a triangulation of $P(n,k)$ with a factorwise-dual-Sperner labeling $\ell$, then it has an equivalent labeling $\ell'$ with an elementary simplex $Q$ in $T$, such that the set of labels 
$\Lambda(Q)=\{\ell'(v) \mid v\in V(Q)\}$ 
is balanced with respect to $V=V_1\sqcup \dots \sqcup V_k$, where 
$V_i=[n]$ for all $1\le i\le k$. 
}
\begin{proof}
We will construct an equivalent Sperner labeling $\ell'$ of $T$.
First for any $J \subset [n]$, let $J+1$ denote the set $\{ i+1 \mid i \in J\}$, where for the purposes of this definition we set $n+1=1$, so that $J+1$ cycles the indices in $J$ by shifting each by $1$.
Note that if $J\neq [n]$, then $(J+1) \setminus J \neq \emptyset$.

Given $v \in V(T)$, set $\ell'_i(v) = \ell_i(v)$ for all $i$ such that $\free(v_i)=\emptyset$. For all other $i$, set $J_i=\supp(v_i)$ and set $\ell'_i(v) = j_i+1$ for some $j_i+1 \in (J_i+1) \setminus J_i$.
This choice clearly satisfies $\ell'_i(v)\in \free(v_i)$.  Then $\ell'$ is a factorwise-dual-Sperner labeling that is equivalent to $\ell$.

We claim that $\ell'$ is a Sperner labeling of $T$. To prove this, we have to show that (a) the vertices of $P(n,k)$ receive pairwise distinct labels in $\ell'$, and (b) the label $\ell'(v)$ of  $v\in V(T)$ matches one of the labels of the vertices of $\supp(v)$ in $P(n,k)$. To see (a), observe that for any vertex $w$ of $P(n,k)$ and any $i\in [k]$, the set $J_i=\supp(w_i)$ is a singleton, and any two vertices of $P(n,k)$ will differ in $J_i$ for at least one $i$. For (b), note that a point $v\in P(n,k)$ is on a face spanned by a vertex $w$ if and only if $\supp(w_i) \subset \supp(v_i)$ for all $i$.
Now, given $v \in V(T)$ let $w=(w_1,\dots,w_k)$ be the vertex of $P(n,k)$ where $\supp(w_i)=\{\ell'_i(v)-1\}$ for each $i$. Then by our definition of $\ell'$ we have  $\ell'_i(w)=(\ell'_i(v)-1)+1=\ell'_i(v)$ and $\ell_i'(v)-1 \in \supp(v_i)$. Thus, $v$ and $w$ have the same label and $v$ lies on a face spanned by $w$.

Thus, by Theorem \ref{polytopal}, there is an elementary simplex $Q$ in $T$ whose image under the piecewise linear map $f$ of Theorem \ref{polytopal} contains the barycenter of $P(n,k)$. The labels of the vertices of $Q$ therefore form a balanced set with respect to the vertex set $V$. 
\end{proof}

Combining Theorem \ref{dualPsperner} with Corollary \ref{lovaszcor} we have: 
\begin{proposition}\label{dualkomiyacover}
Let $T$ be a complete triangulation of $P(n,k)$ with 
factorwise-dual-Sperner labeling $\ell:V(T)\rightarrow [n]^k$. 
Then there exists an equivalent labeling $\ell'$, 
an elementary simplex $Q$ in $T$, and a subset $U \subset V(Q)$ of size:
 \begin{itemize}
  \item $|U|= n$ when $k=2$, or
  \item $|U|\le n(1+\ln k)$ when $k \geq 2$,
 \end{itemize}
such that the set $\Lambda(U)=\{ \ell'(u)=(\ell'_1(u),\dots, \ell'_k(u)) \mid u\in U\}$ has the property that for every $i \in [n]$ and $j\in [k]$ there exists $u\in U$ for which $\ell'_j (u)=i$. 
\end{proposition}
\begin{proof}
Given $\ell$, by Theorem \ref{dualPsperner} there exist
an equivalent Sperner labeling $\ell'$ and an elementary simplex $Q$ in $T$ such that the set of labels 
$\Lambda(Q)=\{ \ell'(v) \mid v\in V(Q)\}$ 
is balanced with respect to $V=V_1\sqcup\dots \sqcup V_k$. Therefore, 
the hypergraph $H=(V,\Lambda(Q))$ has a perfect fractional matching $f:\Lambda(Q)\to [0,1]$. Since $H$ is $k$-uniform, we have that $\nu^*(H) = kn/k = n$.   

Equivalently, in the dual hypergraph $H^D$, $f$ is a perfect fractional cover, and $\tau^*(H^D) = n$. By Corollary \ref{lovaszcor}, we obtain $\tau(H^D)\le (1+\ln k)\tau^*(H^D) = (1+\ln k)n$, and if $k=2$, then by Theorem \ref{bipgallai}, $\tau(H^D)= \tau^*(H^D) = n$. Thus the hypergraph $H^D$ on vertex set $\Lambda(Q)$ and edges $V$ has a cover of size $t$ where $t\le (1+\ln k)n$, and moreover, $t=n$ when $k=2$. By the definition of $H^D$, this means that there exists a set of vertices $U\subset V(Q)$ of size $|U|=t$ such that the set of labelings of $U$ cover each $v\in V$. That is,    
 for every $j\in [k]$ and $i \in V_j=[n]$ there exists $u\in U$ for which $\ell'_j(u)=i$.    
\end{proof}

We now prove our last main theorem. 
\medskip

\noindent {\bf Theorem \ref{main3}} \emph{
Suppose that a set of $p\geq k(n-1)+1$ employees satisfies the following conditions: 
\begin{enumerate}
\item For any partition of $k$ days into $n$ shift each, every employee finds at least one $k$-shift selection acceptable. 
\item The employees prefer empty shifts, if available. 
\item All employee preference sets are closed.   
\end{enumerate}
Then there exists a partition of $k$ days into $n$ shifts each for which a subset of at most $n(1+\ln k)$ employees cover all the $kn$ shifts. Moreover, if $k=2$ there exists a subset of $n$ employees that cover all the $2n$ shifts. } 

\begin{proof}
Fix a subset $S$ of $k(n-1)+1$ players. 
By Lemma \ref{complete}, there exists a complete triangulation $T$ of $P=P(n,k)$. Let $o_v\in S$ be the owner of $v\in V(T)$ in a complete assignment. 

Every vertex $v=(v_1,\dots,v_k)\in V(T)$  corresponds to an partition of the $k$ days into $n$ shifts each. Let the labeling 
$\ell(v)$
be the shift selection chosen by $o_v$: one shift on each day. 
Since $o_v$ prefers empty shifts, for each $j$ we have $\ell_j(v) \in \free(v_j)$ whenever $\free(v_j) \neq \emptyset$. When more than one empty shift exists on a particular day, we can reassign the label for that day if needed, as in the proof of Theorem \ref{dualPsperner}, to obtain a factorwise-dual-Sperner labeling $\ell'$. Thus, by Proposition \ref{dualkomiyacover}, there exists an elementary simplex $Q$ in $T$
and  $U \subset V(Q)$ of size  $|U|\le n(1+\ln k)$ (or $|U|= n$ if $k=2$), such that the set $\{\ell'(u) \mid u\in U\}$ has the property that for every $i \in [n]$ and $j\in [k]$ there exists $u\in U$ for which $\ell'_j(u)$ is $i$. This corresponds to a cover of the shifts.   

A similar argument to the one made in the proof of Theorem \ref{main1} now shows the existence of a single partition of the days into shifts in which the $k$-shift selection of at most of $n(1+\ln k)$ employees (or $n$ employees if $k=2$) cover every shift.
\end{proof}

\section{Acknowledgments}
We thank Florian Frick for valuable conversations. This work was completed while the authors were in residence at at the Mathematical Sciences Research Institute in Berkeley, California, during the Fall 2017 semester where they were supported by the National Science Foundation under Grant No. DMS-1440140.

\end{document}